\documentclass{conm-p-l}
\usepackage{amssymb, stmaryrd}

\newtheorem{theorem}[equation]{Theorem}

\newtheorem{lemma}[equation]{Lemma}

\newtheorem{corollary}[equation]{Corollary}

\newtheorem{proposition}[equation]{Proposition}

\newtheorem{remark}[equation]{Remark}

\newtheorem{definition}[equation]{Definition}

\newtheorem{example}[equation]{Example}

\newcommand{\cM}{\mathcal{M}}

\renewcommand{\gg}{\mathfrak{g}}
\newcommand{\gp}{\mathfrak{p}}
\newcommand{\gq}{\mathfrak{q}}

\makeatletter
\def\revddots{\mathinner{\mkern1mu\raise\p@
\vbox{\kern7\p@\hbox{.}}\mkern2mu
\raise4\p@\hbox{.}\mkern2mu\raise7\p@\hbox{.}\mkern1mu}}
\makeatother

\def\ga{\mathfrak{a}}

\def\gg{\mathfrak{g}}
\def\ggl{\mathfrak{gl}}

\def\gk{\mathfrak{k}}
\def\gl{\mathfrak{l}}
\def\gm{\mathfrak{m}}
\def\gn{\mathfrak{n}}

\def\gp{\mathfrak{p}}
\def\gq{\mathfrak{q}}
\def\gr{\mathfrak{r}}
\def\gs{\mathfrak{s}}
\def\gsl{\mathfrak{sl}}
\def\gso{\mathfrak{so}}
\def\gsu{\mathfrak{su}}
\def\gsp{\mathfrak{sp}}
\def\gt{\mathfrak{t}}
\def\gu{\mathfrak{u}}

\def\gz{\mathfrak{z}}
\def\gS{\mathfrak{S}}

\def\Aut{{\rm Aut}}

\def\Ad{{\rm Ad}}

\def\trace{{\rm trace\,\,}}

\def\Ind{{\rm Ind\,}}

\def\C{\mathbb{C}}
\def\D{\mathbb{D}}
\def\E{\mathbb{E}}
\def\F{\mathbb{F}}

\def\H{\mathbb{H}}

\def\N{\mathbb{N}}

\def\R{\mathbb{R}}

\def\Z{\mathbb{Z}}

\def\cC{\mathcal{C}}

\def\cF{\mathcal{F}}

\def\cH{\mathcal{H}}

\def\cM{\mathcal{M}}

\def\cU{\mathcal{U}}

\begin{document}

\title[Principal series representations of infinite dimensional Lie
groups]{Principal series representations of infinite dimensional Lie \\
groups, II: Construction of induced representations}

\author{Joseph A. Wolf}

\address{Department of Mathematics, University of California,
Berkeley, CA 94720--3840}
\curraddr{}
\email{jawolf@math.berkeley.edu}
\thanks{Research partially supported by the Simons Foundation}

\date{October 18, 2012}

\begin{abstract}We study representations of the classical infinite 
dimensional real simple Lie groups $G$ induced from factor representations 
of minimal parabolic subgroups $P$\,.
This makes strong use of the recently developed structure theory for those 
parabolic subgroups and subalgebras.  In general parabolics in the infinite 
dimensional classical Lie groups are are somewhat more complicated than
in the finite dimensional case, and are not direct limits of finite 
dimensional parabolics.  We extend their structure theory and use it for 
the infinite dimensional analog of the classical principal series 
representations.  In order to do this we examine two types of conditions
on $P$: the flag-closed condition and minimality.  We use some
riemannian symmetric space theory to prove that if
$P$ is flag-closed then any maximal lim-compact subgroup $K$ of $G$
is transitive on $G /P$\,.  When $P$ is minimal we prove that
it is amenable, and we use properties of amenable groups
to induce unitary representations $\tau$ of $P$ up to continuous
representations $\Ind_P^G(\tau)$ of $G$ on complete locally convex 
topological vector spaces.
When $P$ is both minimal and flag-closed we have a decomposition $P = MAN$
similar to that of the finite dimensional case, and we show how this gives 
$K$--spectrum information $\Ind_P^G(\tau)|_K = \Ind_M^K(\tau|_M)$. 
\end{abstract}

\maketitle

\section{Introduction}\label{sec1} \setcounter{equation}{0}

This paper continues a program of extending aspects of representation
theory from finite dimensional real semisimple groups to infinite dimensional 
real Lie groups.  The
finite dimensional theory depends on the structure of parabolic
subgroups.  That structure was recently been worked out for the classical
real direct limit Lie algebras such as $\gsl(\infty,\R)$ and $\gsp(\infty;\R)$ 
\cite{DCPW} and then developed for minimal parabolic subgroups
(\cite{W6}, \cite{W7}).  Here we refine that structure theory, and 
investigate it in detail when the flags defining the parabolic consist of
closed (in the Mackey topology) subspaces.  Then we develop a notion of 
induced representation that makes use of the structure of minimal parabolics, 
and we use it to construct an infinite dimensional counterpart of the 
principal series representations of finite dimensional real reductive Lie 
groups.
\smallskip

The representation theory of finite dimensional real reductive Lie groups 
is based on the now--classical constructions and Plancherel Formula of 
Harish--Chandra.  Let $G$ be a real reductive Lie group of Harish-Chandra 
class, e.g.  $SL(n;\R)$, $U(p,q)$, $SO(p,q)$, \dots.  Then
one associates a series of representations to each conjugacy class of
Cartan subgroups.  Roughly speaking this goes as follows.
Let $Car(G)$ denote the set of conjugacy classes $[H]$ of Cartan subgroups 
$H$ of $G$.  Choose $[H] \in Car(G)$, $H \in [H]$, and an irreducible unitary 
representation $\chi$ of $H$.  Then we have a ``cuspidal'' parabolic subgroup 
$P$ of $G$ constructed from $H$, and a unitary representation 
$\pi_\chi$ of $G$ constructed from $\chi$ and $P$.  Let
$\Theta_{\pi_\chi}$ denote the distribution character of $\pi_\chi$\, .
The Plancherel Formula: if $f\in \cC(G)$, the Harish-Chandra Schwartz space,
then
\begin{equation}\label{planch}
f(x) = \sum_{[H] \in Car(G)} \,\,\int_{\widehat{H}}
        \Theta_{\pi_\chi}(r_xf) d\mu_{[H]}(\chi)
\end{equation}
where $r_x$ is right translation and $\mu_{[H]}$ is Plancherel measure
on the unitary dual $\widehat{H}$.
\smallskip

In order to extend elements of this theory to real semisimple direct
limit groups, we have to look more closely at the construction of the
Harish--Chandra series that enter into (\ref{planch}).
\smallskip

Let $H$ be a Cartan subgroup of $G$.  It is stable under a
Cartan involution $\theta$, an involutive automorphism of $G$ whose
fixed point set $K = G^\theta$ is a maximal compactly embedded\footnote{A
subgroup of $G$ is {\sl compactly embedded} if it has compact image under
the adjoint representation of $G$.} subgroup.
Then $H$ has a $\theta$--stable decomposition
$T \times A$ where $T = H \cap K$ is the compactly embedded part 
and (using lower case Gothic letters for Lie algebras)
$\exp : \ga \to A$ is a bijection.  Then $\ga$ is commutative and acts
diagonalizably on $\gg$.  Any choice of positive $\ga$--root
system defines a parabolic subalgebra $\gp = \gm + \ga + \gn$ in $\gg$ 
and thus defines a parabolic subgroup $P = MAN$ in $G$.  If $\tau$ is an
irreducible unitary representation of $M$ and $\sigma \in \ga^*$ then
$\eta_{\tau,\sigma}: man \mapsto e^{i\sigma(\log a)}\tau(m)$ is a well
defined irreducible unitary representation of $P$.  The equivalence class
of the unitarily induced representation $\pi_{\tau,\sigma} =
\Ind_P^G(\eta_{\tau,\sigma})$ is independent of the choice of positive 
$\ga$--root system.  The group $M$ has (relative) discrete series 
representations,
and $\{\pi_{\tau,\sigma} \mid \tau \text{ is a discrete series rep of } M\}$
is the series of unitary representations associated to 
$\{\Ad(g)H \mid g \in G\}$.
\smallskip

Here we work with the simplest of these series, the case where $P$ is a minimal
parabolic subgroup of $G$, for the classical infinite dimensional real simple
Lie groups $G$.  In \cite{W7} we worked out the basic structure of those minimal
parabolic subgroups.
As in the finite dimensional case, a minimal parabolic has structure 
$P = MAN$ where $M = P \cap K$ is a (possibly infinite)
direct product of torus groups, compact classical groups such as $Spin(n)$,
$SU(n)$, $U(n)$ and $Sp(n)$, and their classical direct limits $Spin(\infty)$,
$SU(\infty)$, $U(\infty)$ and $Sp(\infty)$ (modulo intersections and discrete 
central subgroups).  There in \cite{W7} we also discussed various classes 
of representations
of the lim-compact group $M$ and the parabolic $P$.  Here we discuss the 
unitary induction procedure $\Ind_{MAN}^G(\tau \otimes e^{i\sigma})$ where
$\tau$ is a unitary representation of $M$ and $\sigma \in \ga^*$.
The complication, of course, is that we can no longer integrate over $G/P$.
\smallskip

There are several new ideas in this note.  One is to define a new class of
parabolics, the {\em flag-closed parabolics}, and apply some riemannian 
geometry to prove a transitivity theorem, Theorem \ref{kgp-global}.
Another is to extend the standard finite dimensional decomposition $P = MAN$ 
to minimal parabolics; that is Theorem \ref{lang-alg}.  A third is to put
these together with amenable group theory to construct an analog of induced
representations in which integration over $G/P$ is replaced by a right
$P$--invariant means on $G$.  That produces continuous representations of
$G$ on complete locally convex topological vector spaces, which are
the analog of principal series representations.  Finally, if $P$ is 
flag-closed and minimal, a close look at
this amenable induction process gives the $K$-spectrum of our representations.
representations.
\smallskip

We sketch the nonstandard part of the necessary background in Section 
\ref{sec2abc}.
First, we recall the classical simple real direct limit 
Lie algebras and Lie groups.  There are no surprises.  Then we
sketch the theory of complex and real parabolic subalgebras.  Finally  
we indicate structural aspects such as Levi components and 
the Chevalley decomposition.  That completes the background.
\smallskip

In Section \ref{sec3} we specialize to parabolics whose defining flags consist 
of closed subspaces in the Mackey topology, that is $F = F^{\perp\perp}$.
The main result, Theorem \ref{kgp-global}, is that a maximal 
lim--compact subgroup
$K \subset G$ is transitive on $G/P$.  This involves the geometry of the
(infinite dimensional) riemannian symmetric space $G/K$.  Without the
flag--closed property it would not even be clear whether $K$ has an open orbit
on $G/P$.
\smallskip

In Section \ref{sec4} we work out the basic properties of minimal 
self--normalizing parabolic subgroups of $G$, refining results of
\cite{W6} and \cite{W7}.  The the Levi components are locally
isomorphic to direct sums in an explicit way of subgroups that are either
the compact classical groups $SU(n)$, $SO(n)$ or $Sp(n)$, or
their limits $SU(\infty)$, $SO(\infty)$ or $Sp(\infty)$.  The Chevalley
(maximal reductive part) components are slightly more complicated, for
example involving extensions $1 \to SU(*) \to U(*) \to T^1 \to 1$ as
well as direct products with tori and vector groups.
The main result, Theorem \ref{lang-alg},
is the minimal parabolic analog of standard structure theory for real 
parabolics in finite dimensional real reductive Lie groups.
Proposition \ref{construct-p} then gives an explicit
construction for a self-normalizing flag-closed 
minimal parabolic with a given Levi factor.
\smallskip

In Section \ref{sec5} we put all this together with amenable group
theory.  Since strict direct limits of amenable groups are amenable, 
our maximal lim-compact group $K$ and minimal parabolic subgroups $P$
are amenable.  In particular there are means on
$G/P$, and we consider the set $\cM(G/P)$ of all such means.
Given a homogeneous hermitian vector bundle $\E_\tau \to G/P$, 
we construct a continuous representation $\Ind_P^G(\tau)$
of $G$. The representation space is a complete locally convex topological
vector space, completion of the space of all right uniformly continuous
bounded sections of $\E_\tau \to G/P$.  These representations form the 
{\em principal series} for our real group $G$ and choice of
parabolic $P$.  In the flag-closed case 
we also obtain the $K$-spectrum.
\smallskip

In fact we carry out this ``amenably induced representation'' construction 
somewhat more generally: whenever we have a topological group $G$, a closed 
amenable subgroup $H$ and a $G$--invariant subset of $\cM(G/H)$.
\smallskip

We have been somewhat vague about the unitary representation $\tau$ of $P$.
This is discussed, with references, in \cite{W7}.  We go into it in more
detail in an Appendix. 
\smallskip

I thank Elizabeth Dan-Cohen for pointing out the result indicated below as
Proposition \ref{closed-orthocomp},  and Gestur \' Olafsson for fruitful
discussions on invariant means.

\section{Parabolics in Finitary Simple Real Lie Groups}
\label{sec2abc}
\setcounter{equation}{0}
In this section we sketch the real simple countably infinite dimensional
locally finite (``finitary'') Lie algebras and the corresponding Lie groups,  
following results from \cite{B}, \cite{BS} and \cite{DCPW}.  
Then we recall the structure of parabolic subalgebras of
the complex Lie algebras $\gg_\C$ = 
$\ggl(\infty;\C)$, $\gsl(\infty);\C)$, $\gso(\infty;\C)$ and $\gsp(\infty;\C)$.
Next, we indicate the structure of real parabolic subalgebras, in other words
parabolic subalgebras  of real forms of those algebra $\gg_\C$.  This summarizes
results from \cite{DC1}, \cite{DC2} and \cite{DCPW}.
\medskip

\subsection*{2A. Finitary Simple Real Lie Groups.}
The three classical simple locally finite countable--dimensional 
complex Lie algebras are the classical direct limits
$\gg_\C = \varinjlim \gg_{n,\C}$ given by
\begin{equation}\label{cpx-lie}
\begin{aligned}
& \gsl(\infty,\C) = \varinjlim \gsl(n;\C),\\
& \gso(\infty,\C) = \varinjlim \gso(2n;\C) = \varinjlim \gso(2n+1;\C), \\
& \gsp(\infty,\C) = \varinjlim \gsp(n;\C),
\end{aligned}
\end{equation}
where the direct systems are given by the inclusions of the form
$A \mapsto (\begin{smallmatrix} A & 0 \\ 0 & 0 \end{smallmatrix} )$.
We will also consider the locally reductive algebra
$\ggl(\infty;\C) = \varinjlim \ggl(n;\C)$ along with $\gsl(\infty;\C)$.
The direct limit process of (\ref{cpx-lie}) defines the universal enveloping
algebras
\begin{equation}\label{univ-cpx-lie}
\begin{aligned}
& \cU(\gsl(\infty,\C)) = \varinjlim \cU(\gsl(n;\C)) \text{ and }
	\cU(\ggl(\infty,\C)) = \varinjlim \cU(\ggl(n;\C)), \\
& \cU(\gso(\infty,\C)) = \varinjlim \cU(\gso(2n;\C)) = \varinjlim \cU(\gso(2n+1;\C)),        
        \text{ and }\\
& \cU(\gsp(\infty,\C)) = \varinjlim \cU(\gsp(n;\C)),
\end{aligned}
\end{equation}
\smallskip

Of course each of these Lie algebras $\gg_\C$ has the underlying structure of
a real Lie algebra.  Besides that, their real forms are as follows (\cite{B},
\cite{BS}, \cite{DCPW}).
\smallskip

If  $\gg_\C = \gsl(\infty;\C)$,  then $\gg$ is one of 
$\gsl(\infty;\R) = \varinjlim \gsl(n;\R)$, the real
special linear Lie algebra;
$\gsl(\infty;\H) = \varinjlim \gsl(n;\H)$, the quaternionic
special linear Lie algebra, given by
$\gsl(n;\H) := \ggl(n;\H) \cap \gsl(2n;\C)$;
$\gs\gu(p,\infty) = \varinjlim \gs\gu(p,n)$, the complex special
unitary Lie algebra of real rank $p$; or
$\gs\gu(\infty,\infty) = \varinjlim \gs\gu(p,q)$, complex
special unitary algebra of infinite real rank.
\smallskip

If  $\gg_\C = \gso(\infty;\C)$,  then $\gg$ is one of 
$\gso(p,\infty) = \varinjlim \gso(p,n)$, the real orthogonal
Lie algebra of finite real rank $p$;
$\gso(\infty,\infty) = \varinjlim \gso(p,q)$, the real
orthogonal Lie algebra of infinite real rank; or
$\gso^*(2\infty) = \varinjlim \gso^*(2n)$
\smallskip

If  $\gg_\C = \gsp(\infty;\C)$,  then $\gg$ is one of 
$\gsp(\infty;\R) = \varinjlim \gsp(n;\R)$, the real
symplectic Lie algebra;
$\gsp(p,\infty) = \varinjlim \gsp(p,n)$, the quaternionic unitary
Lie algebra of real rank $p$; or
$\gsp(\infty,\infty) = \varinjlim \gsp(p,q)$, quaternionic
unitary Lie algebra of infinite real rank.
\smallskip

If  $\gg_\C = \ggl(\infty;\C)$,  then $\gg$ is one 
$\ggl(\infty;\R) = \varinjlim \ggl(n;\R)$, the real
general linear Lie algebra;
$\ggl(\infty;\H) = \varinjlim \ggl(n;\H)$, the quaternionic
general linear Lie algebra;
$\gu(p,\infty) = \varinjlim \gu(p,n)$, the complex unitary
Lie algebra of finite real rank $p$; or
$\gu(\infty,\infty) = \varinjlim \gu(p,q)$, the complex
unitary Lie algebra of infinite real rank.
\smallskip

As in (\ref{univ-cpx-lie}), given one of these Lie algebras 
$\gg = \varinjlim \gg_n$ we have the universal enveloping algebra.
Just as in the
finite dimensional case, we use the universal enveloping algebra of the
complexification.  Thus when we write $\cU(\gg)$ it is understood that 
we mean $\cU(\gg_\C)$.  
\smallskip

The corresponding Lie groups are exactly what one expects.  First the
complex groups, viewed either as complex groups or as real groups,
\begin{equation}\label{cpl_gps}
\begin{aligned}
& SL(\infty;\C) = \varinjlim SL(n;\C) \text{ and } GL(\infty;\C) = 
     \varinjlim GL(n;\C), \\
& SO(\infty;\C) = \varinjlim SO(n;\C) = \varinjlim SO(2n;\C)
     = \varinjlim SO(2n+1;\C), \\
& Sp(\infty;\C) = \varinjlim Sp(n;\C).
\end{aligned}
\end{equation}
The real forms of the complex special and general linear groups 
$SL(\infty;\C)$ and $GL(\infty;\C)$ are
\begin{equation}\label{typeA}
\begin{aligned}
& SL(\infty;\R) \text{ and } GL(\infty;\R): 
     \text{ real special/general linear groups, } \\
& SL(\infty;\H): \text{ quaternionic special linear group, } \\
& SU(p,\infty): \text{ special unitary groups
     of real rank } p < \infty, \\
& SU(\infty,\infty): \text{ 
     unitary groups of infinite real rank,}\\
& U(p,\infty): \text{ unitary groups
     of real rank } p < \infty, \\
& U(\infty,\infty): \text{  
     unitary groups of infinite real rank.}
\end{aligned}
\end{equation}
The real forms of the complex orthogonal and spin groups 
$SO(\infty;\C)$ and $Spin(\infty;\C)$ are
\begin{equation}\label{typeBD}
\begin{aligned}
& SO(p,\infty) \text{, } Spin(p;\infty): \text{ orthogonal/spin
     groups of real rank } p < \infty,  \\
& SO(\infty,\infty) \text{, } Spin(\infty,\infty): \text{ orthogonal/spin
     groups of real rank } \infty, \\
& SO^*(2\infty) = \varinjlim SO^*(2n), \text{ which doesn't have a 
	standard name}
\end{aligned}
\end{equation}
Here $SO^*(2n) = SO(2n;\C) \cap U(n,n)$ where $SO^*(2n)$ is 
defined by the form $\kappa(x,y) := \sum x^\ell i \bar y^\ell = {^tx} i \bar y$
and $SO(2n;\C)$ is defined by $(u,v) = \sum (u_j v_{n+jr} + v_{n+j}w_j)$.
Finally, the real forms of the complex symplectic group $Sp(\infty;\C)$ are
\begin{equation}\label{typeC}
\begin{aligned}
& Sp(\infty;\R): \text{ real symplectic group,} \\
& Sp(p,\infty): \text{ quaternion unitary group of real rank } p < \infty, \text{ and }\\
& Sp(\infty,\infty): \text{ quaternion unitary group of infinite real rank.}
\end{aligned}
\end{equation}

\subsection*{2B. Parabolic Subalgebras.}
For the structure of parabolic subalgebras we must 
describe $\gg_\C$ in terms of linear spaces.  Let $V_\C$ and $W_\C$ be 
nondegenerately paired countably infinite dimensional complex vector spaces.   
Then
$\ggl(\infty,\C) = \ggl(V_\C, W_\C) := V_\C\otimes W_\C$ consists of all 
finite linear combinations of the rank $1$ operators 
$v\otimes w: x \mapsto \langle w, x\rangle v$.  In
the usual ordered basis of $V_\C = \C^\infty$, parameterized by the positive 
integers, and with the dual basis of $W_\C = V_\C^* = (\C^\infty)^*$, we can
view  $\ggl(\infty,\C)$ can be viewed as infinite matrices with only
finitely many nonzero entries.  However $V_\C$ has more exotic ordered bases, 
for
example parameterized by the rational numbers, where the matrix picture is
not intuitive.
\smallskip

The rank 1 operator $v\otimes w$ has a well defined trace, so trace is
well defined on $\ggl(\infty,\C)$.  Then $\gsl(\infty,\C)$ is the
traceless part, $\{g \in \ggl(\infty;\C) \mid \trace g = 0\}$.  
\smallskip

In the orthogonal case we can take $V_\C = W_\C$ using the symmetric bilinear
form that defines $\gso(\infty;\C)$.  Then
$$
\gso(\infty;\C) = \gso(V,V) = \Lambda\ggl(\infty;\C) \text{ where }
	\Lambda(v\otimes v') = v\otimes v' - v'\otimes v.
$$
In other words, in an ordered orthonormal basis of $V_\C = \C^\infty$ 
parameterized by the positive integers, $\gso(\infty;\C)$ can be viewed
as the infinite antisymmetric matrices with only finitely many nonzero entries.
\smallskip

Similarly, in the symplectic case we can take $V_\C = W_\C$ using the 
antisymmetric
bilinear form that defines $\gsp(\infty;\C)$, and then
$$
\gsp(\infty;\C) = \gsp(V,V) = S\ggl(\infty;\C) \text{ where }
        S(v\otimes v') = v\otimes v' + v'\otimes v.
$$
In an appropriate ordered basis of $V_\C = \C^\infty$ parameterized
by the positive integers, $\gsp(\infty;\C)$ can be viewed
as the infinite symmetric matrices with only finitely many nonzero entries.
\smallskip

In the finite dimensional setting, Borel subalgebra means a maximal solvable
subalgebra, and parabolic subalgebra means one that contains a Borel.  It is
the same here except that one must use {\em locally solvable} to avoid the
prospect of an infinite derived series.

\begin{definition}
{\rm A maximal locally solvable subalgebra of $\gg_\C$ is called a 
{\em Borel subalgebra} of $\gg_\C$\,.  A {\em parabolic subalgebra} of $\gg_\C$ 
is a subalgebra that contains a Borel subalgebra.}
\hfill $\diamondsuit$
\end{definition}
\smallskip

In the finite dimensional setting a parabolic subalgebra is the stabilizer
of an appropriate nested sequence of subspaces (possibly with an orientation
condition in the orthogonal group case).  In the infinite dimensional setting
here,  one must be very careful as
to which nested sequences of subspaces are appropriate.  If $F$ is a subspace
of $V_\C$ then $F^\perp$ denotes its annihilator in $W_\C$. Similarly if ${'F}$
is a subspace of $W_\C$ the ${'F}^\perp$ denotes its annihilator in $V_\C$.  
We say that $F$ (resp. ${'F}$) is {\em closed} if $F = F^{\perp\perp}$
(resp. ${'F} = {'F}^{\perp\perp}$).  This is the closure relation in the
Mackey topology \cite{M}, i.e. the weak topology for the functionals on $V_\C$ from
$W_\C$ and on $W_\C$ from $V_\C$.
\smallskip

In order to avoid repeating the following definitions later on, we make them 
in somewhat greater generality than we need just now.

\begin{definition} \label{genflag}
{\rm Let $V$ and $W$ be countable dimensional vector spaces over a real
division ring
$\D = \R, \C \text{ or } \H$, with a nondegenerate bilinear pairing
$\langle \cdot , \cdot \rangle : V \times W \rightarrow \D$.  
A {\em chain} or {\em $\D$--chain} in $V$ (resp. $W$) is a set of 
$\D$--subspaces totally ordered by inclusion.  An {\em generalized $\D$--flag}
in $V$ (resp. $W$) is an $\D$--chain such that each subspace has an 
immediate predecessor or an immediate successor in the inclusion ordering, 
and every nonzero vector of $V$ (or $W$) is caught between an immediate 
predecessor successor (IPS) pair.  An generalized $\D$--flag $\cF$ in $V$
(resp. ${'\cF}$ in $W$) is {\em semiclosed} if 
$F \in \cF$ with $F \ne F^{\perp\perp}$ implies $\{F,F^{\perp\perp}\}$
is an IPS pair (resp. $'F \in {'\cF}$ with $'F \ne 'F^{\perp\perp}$ 
implies $\{'F,'F^{\perp\perp}\}$ is an IPS pair).}
\hfill $\diamondsuit$
\end{definition}

\begin{definition}\label{taut}
{\rm Let $\D$, $V$ and $W$ be as above.
Generalized $\D$--flags $\cF$ in $V$ and ${'\cF}$ in $W$ form a 
{\em taut couple} when
(i) if $F \in \cF$ then $F^\perp$ is
invariant by the $\ggl$--stabilizer of $'\cF$ and
(ii) if $'F \in {'\cF}$ then its annihilator
$'F^\perp$ is invariant by the $\ggl$--stabilizer of $\cF$.} 
\hfill $\diamondsuit$
\end{definition}

In the
$\gso$ and $\gsp$ cases one can use the associated bilinear form to
identify $V_\C$ with $W_\C$ and $\cF$ with ${'\cF}$.
Then we speak of a generalized flag $\cF$ in $V$ as {\em self--taut}.
If $\cF$ is a self--taut generalized flag in $V$ then \cite{DCPW}
every $F \in \cF$ is either isotropic or co--isotropic.
\smallskip

\begin{theorem} \label{self-norm-cpx-parab}
The self--normalizing parabolic subalgebras of the Lie algebras $\gsl(V,W)$ and
$\ggl(V,W)$ are the normalizers of taut couples of semiclosed generalized flags
in $V$ and $W$, and this is a one to one correspondence.  
The self--normalizing parabolic subalgebras of $\gsp(V)$
are the normalizers of self--taut semiclosed generalized flags in $V$, and 
this too is a one to one correspondence.
\end{theorem}

\begin{theorem} \label{self-norm-cpx-parab-so}
The self--normalizing parabolic subalgebras of $\gso(V)$ are the 
normalizers of self--taut semiclosed generalized flags $\cF$ in $V$, and
there are two possibilities:
\begin{enumerate}
\item the flag $\cF$ is uniquely determined by the parabolic, or
\item there are exactly three self--taut generalized flags with the same 
stabilizer as $\cF$.
\end{enumerate}
The latter case occurs precisely when there exists an isotropic subspace $L \in \cF$ with  $\dim_\C L^\perp / L = 2$.  The three flags with the same stabilizer are then
\begin{itemize}
\item[] $\{F \in \cF \mid F \subset L \textrm{ or } L^\perp \subset F \}$
\item[] $\{F \in \cF \mid F \subset L \textrm{ or } L^\perp \subset F \} \cup M_1$
\item[] $\{F \in \cF \mid F \subset L \textrm{ or } L^\perp \subset F \} \cup M_2$
\end{itemize}
where $M_1$ and $M_2$ are the two maximal isotropic subspaces containing $L$.
\end{theorem}

If $\gp$ is a (real or complex) subalgebra of $\gg_\C$
and $\gq$ is a quotient algebra isomorphic to $\ggl(\infty;\C)$, say
with quotient map $f : \gp \to \gq$, then we refer to the composition
$trace\circ f : \gp \to \C$ as an {\em infinite trace} on $\gg_\C$.  If
$\{f_i\}$ is a finite set of infinite traces on $\gg_\C$ and $\{c_i\}$
are complex numbers, then we refer to the condition $\sum c_if_i = 0$ as
an {\em infinite trace condition} on $\gp$.  
\smallskip

\begin{theorem} \label{gen-cpx-parab}
The parabolic subalgebras  $\gp$ in $\gg_\C$ are the algebras
obtained from self normalizing parabolics $\widetilde{\gp}$ by imposing 
infinite trace conditions.
\end{theorem}

As a general principle one tries to be explicit by constructing representations
that are as close to irreducible as feasible.  For this reason we will be
constructing principal series representations by inducing from parabolic 
subgroups that are minimal among the self--normalizing parabolic subgroups.  
\smallskip

Now we discuss the structure of parabolic subalgebras of
real forms of the classical $\gsl(\infty,\C)$, $\gso(\infty,\C)$,
$\gsp(\infty,\C)$ and $\ggl(\infty,\C)$.  In this section $\gg_\C$ will
always be one of them and $G_\C$ will be the corresponding connected complex
Lie group.  Also, $\gg$ will be a real form of $\gg_\C$, and
$G$ will be the corresponding connected real subgroup of $G_\C$.
\smallskip

\begin{definition}\label{defrealp}
Let $\gg$ be a real form of $\gg_\C$.  Then a subalgebra 
$\gp \subset \gg$ is a {\em parabolic subalgebra} if
its complexification $\gp_\C$ is a parabolic subalgebra of $\gg_\C$.
\hfill $\diamondsuit$
\end{definition}

When $\gg$ has two inequivalent defining representations, in
other words when 
$$
\gg = \gsl(\infty;\R),\,\, \ggl(\infty;\R),\,\, 
  \gsu(*,\infty),\,\, \gu(*,\infty),\,\, 
  \text{ or } \,\,\gsl(\infty;\H)
$$
we denote them by $V$ and $W$, and when $\gg$ has only one 
defining representation, in other words when
$$
\gg = \gso(*,\infty),\,\, \gsp(*,\infty),\,\,
   \gsp(\infty;\R),\,\, \text{ or } \,\,\gso^*(2\infty) 
   \text{ as quaternion matrices,}
$$
we denote it by $V$.  The commuting algebra of $\gg$ on $V$ is a
real division algebra $\D$.  The main result of \cite{DCPW} is

\begin{theorem} \label{real-parab}
Suppose that $\gg$ has two inequivalent defining representations.  Then
a subalgebra of $\gg$ (resp. subgroup of $G$) is
parabolic if and only if it is defined by infinite trace conditions
(resp. infinite determinant conditions) on the
$\gg$--stabilizer (resp. $G$--stabilizer) of
a taut couple of generalized $\D$--flags $\cF$ in $V$ and $'\cF$ in $W$.

Suppose that $\gg$ has only one defining representation.
A subalgebra of $\gg$ (resp. subgroup) of $G$ is
parabolic if and only if it is defined by infinite trace conditions
(resp. infinite determinant conditions) on the
$\gg$--stabilizer (resp. $G$--stabilizer) of
a self--taut generalized $\D$--flag $\cF$ in $V$.
\end{theorem}

\subsection*{2C. Levi Components and Chevalley Decompositions.}
Now we turn to Levi components of complex parabolic subalgebras,
recalling results from \cite{DP1}, \cite{DP2}, \cite{DC2}, \cite{DP3},
\cite{DC3} and \cite{W6}.  We start with the definition.

\begin{definition}\label{levi}
{\em Let $\gp_\C$ be a locally finite Lie algebra
and $\gr_\C$ its locally solvable radical.  A subalgebra $\gl_\C \subset \gp_\C$ is a
{\em Levi component} if $[\gp_\C,\gp_\C]$ is the semidirect sum
$(\gr_\C \cap [\gp_\C,\gp_\C]) \subsetplus \gl_\C$.}\hfill $\diamondsuit$
\end{definition}

Every finitary Lie algebra has a Levi
component.  Evidently, Levi components are maximal semisimple subalgebras,
but the converse fails for finitary Lie algebras.  In any case,
parabolic subalgebras of our classical Lie algebras $\gg_\C$ have maximal 
semisimple subalgebras, and those are their Levi components.

\begin{definition}\label{standard}
{\em Let $X_\C \subset V_\C$ and $Y_\C \subset W_\C$ be paired subspaces, isotropic in the
orthogonal and symplectic cases.  The subalgebras 
$$
\begin{aligned}
&\ggl(X_\C,Y_\C) \subset \ggl(V_\C,W_\C) \phantom{an}\text{ and }  \gsl(X_\C,Y_\C) \subset \gsl(V_\C,W_\C),\\
&\Lambda \ggl(X_\C,Y_\C) \subset \Lambda \ggl(V_\C,V_\C) \text{ and }
S\ggl(X_\C,Y_\C) \subset S\ggl(V_\C,V_\C)
\end{aligned}
$$ 
are called {\em standard}.}\hfill $\diamondsuit$
\end{definition}

\begin{proposition}\label{struc-levi}
A subalgebra $\gl_\C \subset \gg_\C$ is the Levi
component of a parabolic subalgebra of $\gg_\C$
if and only if it is the direct sum of standard special linear
subalgebras and at most one subalgebra $\Lambda \ggl(X_\C,Y_\C)$ in the orthogonal case,
at most one subalgebra $S\ggl(X_\C,Y_\C)$ in the symplectic case.
\end{proposition}

The occurrence of ``at most one subalgebra'' in Proposition \ref{struc-levi}
is analogous to the finite dimensional case, where it is seen by deleting some
simple root nodes from a Dynkin diagram.
\smallskip

Let $\gp_\C$ be the parabolic subalgebra of $\gsl(V_\C, W_\C)$ or $\ggl(V_\C, W_\C)$ defined
by the taut couple $(\cF, {'\cF})$ of semiclosed generalized flags.  Denote
\begin{equation}\label{J}
\begin{aligned}
&J = \{(F',F'') \text{ IPS pair in } \cF \mid F' = (F')^{\perp\perp}
	\text{ and } \dim F''/F' > 1\},\\
&'J = \{('F',{'F''}) \text{ IPS pair in } {'\cF} \mid {'F}' = ('F')^{\perp\perp},
         \dim {'F''}/{'F'} > 1\}.
\end{aligned}
\end{equation}
Since $V_\C \times W_\C \to \C$ is nondegenerate the sets $J$ and $'J$ are in
one to one correspondence by: $(F''/F') \times ({'F''}/{'F'}) \to \C$ is
nondegenerate.  We use this to identify $J$ with $J'$, and we write
$(F_j',F_j'')$ and $('F_j',{'F_j''})$ treating $J$ as an index set.

\begin{theorem}\label{glpar}
Let $\gp_\C$ be the parabolic subalgebra of $\gsl(V_\C, W_\C)$ or $\ggl(V_\C, W_\C)$ defined
by the taut couple $\cF$ and $'\cF$ of semiclosed generalized flags.  For each
$j \in J$ choose a subspace $X_{j,\C} \subset V_\C$ and a subspace $Y_{j,\C} \subset W_\C$
such that $F_j'' = X_{j,\C} + F_j'$ and $'F_j'' = Y_{j,\C} + {'F_j}'$
Then $\bigoplus_{j \in J}\, \gsl(X_{j,\C},Y_{j,\C})$ is a Levi component of $\gp_\C$. The
inclusion relations of $\cF$ and $'\cF$ induce a total order on $J$.

Conversely, if $\gl_\C$ is a Levi component of $\gp_\C$ then there exist subspaces
$X_{j,\C} \subset V_\C$ and $Y_{j,\C} \subset W_\C$ such that 
$\gl = \bigoplus_{j \in J}\, \gsl(X_{j,\C},Y_{j,\C})$.
\end{theorem}

Now the idea of finite matrices with blocks down the diagonal suggests the
construction of $\gp_\C$ from the totally ordered set $J$ and the 
Lie algebra direct sum
$\gl_\C = \bigoplus_{j \in J}\, \gsl(X_{j,\C},Y_{j,\C})$ of standard special linear 
algebras.  We outline the idea of the construction; see \cite{DC3}.  First, 
$\langle X_{j,\C}, Y_{j',\C}\rangle = 0$ for $j \ne j'$ because the
$\gl_j = \gsl(X_{j,\C},Y_{j,\C})$ commute with each other.  Define
$U_{j,\C} := (( \bigoplus_{k \leqq j}\, X_{k,\C})^\perp \oplus Y_{j,\C})^\perp$. 
Then one proves $U_{j,\C} = ((U_{j,\C} \oplus X_{j,\C})^\perp \oplus Y_{j,\C})^\perp$.
From that, one shows that there is a unique semiclosed generalized flag 
$\cF_{min}$ in $V_\C$ with the same stabilizer as the set 
$\{U_{j,\C}, U_{j,\C} \oplus X_{j,\C}\,  |\,  j \in J\}$. 
One constructs similar subspaces $'U_{j,\C} \subset W_\C$ and shows that there is a 
unique semiclosed generalized flag $'\cF_{min}$ in $W_\C$ with the same stabilizer 
as the set $\{'U_{j,\C}, {'U}_{j,\C} \oplus Y_{j,\C} \, |\,  j \in J\}$.  In fact
$(\cF_{min} , {'\cF}_{min})$ is the minimal taut couple with IPS
pairs $U_{j,\C} \subset (U_{j,\C} \oplus X_{j,\C})$ in $\cF_{min}$ and 
$(U_{j,\C} \oplus X_{j,\C})^\perp \subset ((U_{j,\C} \oplus X_{j,\C})^\perp \oplus Y_{j,\C})$ in ${'\cF}_{min}$ 
for $j \in J$.  If $(\cF_{max}, {'\cF}_{max})$ is maximal among the taut couples of
semiclosed generalized flags with IPS pairs $U_{j,\C} \subset (U_{j,\C} \oplus X_{j,\C})$ 
in $\cF_{max}$
and $(U_{j,\C} \oplus X_{j,\C})^\perp \subset ((U_{j,\C} \oplus X_{j,\C})^\perp \oplus Y_{j,\C})$ in 
${'\cF}_{max}$ then the corresponding parabolic $\gp_\C$ has Levi component $\gl_\C$.
\smallskip

The situation is essentially the same for Levi components of
parabolic subalgebras of $\gg_\C = \gso(\infty;\C) \text{ or }
\gsp(\infty;\C)$, except that we modify the definition (\ref{J}) of $J$ to add
the condition that $F''$ be isotropic, and we add the orientation aspect
of the $\gso$ case.

\begin{theorem}\label{sosplevi}
Let $\gp_\C$ be the parabolic subalgebra of $\gg_\C = \gso(V_\C)$ or $\gsp(V_\C)$,
defined by the self--taut semiclosed generalized flag $\cF$.  Let 
$\widetilde{F}$ be the union of all subspaces $F''$ in IPS pairs 
$(F',F'')$ of $\cF$ for which $F''$ is isotropic.  Let $\widetilde{'F}$ be the 
intersection of all subspaces $F'$ in IPS pairs for which $F'$ is
closed ($F' = (F')^{\perp\perp}$) and coisotropic.
Then $\gl_\C$ is a Levi component of $\gp_\C$ if and
only if there are isotropic subspaces $X_{j,\C}, Y_{j,\C}$ in $V_\C$ such that
$$
\text{$F''_j =F'_j + X_{j,\C}$ and ${'F''_j} ={'F_j} +Y_{j,\C}$ for every $j \in J$}
$$
and a subspace $Z_\C$ in $V_\C$ such that
$\widetilde{F} = Z_\C + \widetilde{'F}$,
where $Z_\C = 0$ in case $\gg_\C = \gso(V_\C)$ and 
$\dim \widetilde{F}/\widetilde{'F} \leqq 2$, such that 
$$
\begin{aligned}
&\gl_\C = \gsp(Z_\C) \oplus {\bigoplus}_{j \in J}\ \gsl(X_{j,\C},Y_{j,\C}) \text{ if }
	\gg_\C = \gsp(V_\C),\\
&\gl_\C = \gso(Z_\C) \oplus {\bigoplus}_{j \in J}\ \gsl(X_{j,\C},Y_{j,\C}) \text{ if }
	\gg_\C = \gso(V_\C).
\end{aligned}
$$
Further, the inclusion relations of $\cF$ induce a 
total order on $J$ which leads to a construction of $\gp_\C$ from $\gl_\C$.
\end{theorem}

Next we describe the Chevalley decomposition for parabolic subalgebras,
following \cite{DC2}.
\smallskip

Let $\gp_\C$ be a locally finite linear Lie algebra, in our case a subalgebra 
of $\ggl(\infty,\C)$.  Every element $\xi \in \gp_\C$ has a Jordan canonical form,
yielding a decomposition $\xi = \xi_{ss} + \xi_{nil}$ into semisimple and
nilpotent parts.  The algebra $\gp_\C$ is {\em splittable} if it contains the 
semisimple and the nilpotent parts of each of its elements.  Note that
$\xi_{ss}$ and $\xi_{nil}$ are polynomials in $\xi$; this follows from the
finite dimensional fact.  In particular, if $X_\C$ is any $\xi$--invariant
subspace of $V_\C$ then it is invariant under both $\xi_{ss}$ and $\xi_{nil}$.
\smallskip

Conversely, parabolic subalgebras (and many others) of our classical 
Lie algebras $\gg_\C$ are splittable.
\smallskip

The {\em linear nilradical} of a subalgebra $\gp_\C \subset \gg_\C$ is the set 
$\gp_{nil,\C}$ of all nilpotent elements of the locally solvable radical $\gr_\C$
of $\gp_\C$.  It is a locally nilpotent ideal in $\gp_\C$ and satisfies
$\gp_{nil,\C} \cap [\gp_\C, \gp_\C] = \gr_\C \cap [\gp_\C, \gp_\C]$.
\smallskip

If $\gp_\C$ is splittable then it has a well defined maximal locally reductive 
subalgebra $\gp_{red,\C}$.  This means that $\gp_{red,\C}$ is an increasing union of
finite dimensional reductive Lie algebras, each reductive in the next.
In particular $\gp_{red,\C}$ maps isomorphically under the projection 
$\gp_\C \to \gp_\C/\gp_{nil,\C}$.  That gives a semidirect sum decomposition
$\gp_\C = \gp_{nil,\C} \subsetplus \gp_{red,\C}$ analogous to the Chevalley
decomposition for finite dimensional algebraic Lie algebras.  Also, here, 
\begin{equation}\label{chev}
\gp_{red,\C} = \gl_\C \subsetplus\gt_\C  \quad \text{ and } \quad [\gp_{red,\C},\gp_{red,\C}] = \gl_\C
\end{equation}
where $\gt_\C$ is a toral subalgebra and $\gl_\C$ is the Levi component of $\gp_\C$.
A glance at $\gu(\infty)$ or $\ggl(\infty;\C)$ shows that the semidirect 
sum decomposition of $\gp_{red,\C}$ need not be direct.
\smallskip

Now we turn to Levi components and Chevalley decompositions for
real parabolic subalgebras in the real classical Lie algebras.  
\smallskip

Let $\gg$ be a real form of a classical locally finite complex simple 
Lie algebra $\gg_\C$.  Consider a real parabolic subalgebra $\gp$.  It
has form $\gp = \gp_\C \cap \gg$ where its complexification $\gp_\C$
is parabolic in $\gg_\C$.  Let $\tau$ denote complex conjugation of
$\gg_\C$ over $\gg$.  Then the locally solvable radical $\gr_\C$ of
$\gp_\C$ is $\tau$--stable because $\gr_\C + \tau\gr_\C$ is a locally
solvable ideal, so the locally solvable radical $\gr$ of $\gp$ is
a real form of $\gr_\C$.  Similarly the linear nilradical $\gn$
of $\gp$ is a real form of the linear nilradical $\gn_\C$ of $\gg_\C$.
\smallskip

Let $\gl$ be a maximal semisimple subalgebra of $\gp$.  Its
complexification $\gl_\C$ is a maximal semisimple subalgebra, hence a Levi
component, of $\gp_\C$.  Thus $[\gp_\C,\gp_\C]$ is the semidirect
sum $(\gr_\C \cap [\gp_\C,\gp_\C]) \subsetplus \gl_\C$.  The elements of
this formula all are $\tau$--stable, so we have proved

\begin{proposition}\label{real-levi-1}
The Levi components of $\gp$ are real forms of the Levi components
of $\gp_\C$.
\end{proposition}

\begin{remark}\label{possible-tau}{\rm
If $\gg_\C$ is $\gsl(V_\C,W_\C)$ or $\ggl(V_\C,W_\C)$ as in 
Theorem \ref{glpar} then we have
$\gl_\C = \bigoplus_{j \in J}\, \gsl(X_{j,\C},Y_{j,\C})$.
Initially the possibilities for the action of $\tau$ are
\begin{itemize}
\item $\tau$ preserves $\gsl(X_{j,\C},Y_{j,\C})$ with fixed point set 
	$\gsl(X_j,Y_j) \cong \gsl(*;\R)$,
\item $\tau$ preserves $\gsl(X_{j,\C},Y_{j,\C})$ with fixed point set
        $\gsl(X_j,Y_j) \cong \gsl(*;\H)$, 
\item $\tau$ preserves $\gsl(X_{j,\C},Y_{j,\C})$ with fixed point set
        $\gsu(X'_j,X_j'') \cong \gsu(*,*)$ where $X_j = X'_j + X_j''$, and
\item $\tau$ interchanges two summands $\gsl(X_{j,\C},Y_{j,\C})$ and
        $\gsl(X_{j',\C},Y_{j',\C})$ of $\gl_\C$, with fixed point set 
        the diagonal ($\cong \gsl(X_{j,\C},Y_{j,\C})$) of their direct sum.
\end{itemize}
If $\gg_\C = \gso(V_\C)$ as in Theorem \ref{sosplevi}, $\gl_\C$ can also have a
summand $\gso(Z_\C)$, or if $\gg_\C = \gsp(V_\C)$ it can also have a summand
$\gsp(V_\C)$.  Except when $A_4 = D_3$ occurs, these additional summands must
be $\tau$--stable, resulting in fixed point sets
\begin{itemize}
\item when $\gg_\C = \gso(V_\C)$: $\gso(Z_\C)^\tau$ is $\gso(*,*)$ or $\gso^*(2\infty)$,
\item when $\gg_\C = \gsp(V_\C)$: $\gsp(Z_\C)^\tau$ is $\gsp(*,*)$ or $\gsp(*;\R)$.
\end{itemize}
And $A_4 = D_3$ cases will not cause problems. \hfill$\diamondsuit$
}
\end{remark}

\section{Parabolics Defined by Closed Flags.}\label{sec3}
\setcounter{equation}{0}

A semiclosed generalized flag $\cF = \{F_\alpha\}_{\alpha \in A}$ is
{\em closed} if all successors in the generalized flag are closed
i.e. if $F''_\alpha = (F''_\alpha)^{\perp\perp}$ for each
immediate predecessor successor (IPS) pair $(F_\alpha',F_\alpha'')$
in $\cF$.  If a complex parabolic $\gp_\C$ is defined by 
a taut couple of closed generalized flags, or by a self dual closed
generalized flag, the we say that $\gp_\C$ is 
{\em flag-closed}.  We say that a real parabolic subalgebra 
$\gp \subset \gg$
is {\em flag-closed} if it is a real form of a flag-closed parabolic subalgebra
$\gp_\C \subset \gg_\C$.  We say ``flag-closed'' for parabolics in
order to avoid confusion later with topological closure.
Theorems 5.6 and 6.6 in the paper \cite{DC2} of E. Dan-Cohen and I. Penkov
tell us

\begin{proposition}\label{closed-orthocomp} 
Let $\gp$ be a parabolic subalgebra of $\gg$ and let $\gn$
denote its linear nilradical.  If $\gp$ is flag-closed, then
$\gp = \gn^\perp$ relative to the bilinear form $(x,y) = \trace(xy)$
on $\gg$.
\end{proposition}

Given  $G = \varinjlim G_n$ acting on $V = \varinjlim V_n$ where
the $d_n = \dim V_n$ are increasing and finite, we have Cartan
involutions $\theta_n$ of the groups $G_n$ such that
$\theta_{n+1}|_{G_n} = \theta_n$, and their limit
$\theta = \varinjlim \theta_n$ (in other words
$\theta|_{G_n} = \theta_n$) is the corresponding Cartan involution
of $G$.  It has fixed point set
$$
K = G^\theta = \varinjlim K_n
$$
where $K_n = G_n^{\theta_n}$ is a maximal compact subgroup of
$G_n$.  We refer to $K$ as a {\em maximal lim-compact subgroup}
of $G$, and to $\gk = \gg^\theta$ as a {\em maximal lim-compact 
subalgebra} of $\gg$\,.  Here, for brevity, we write $\theta$ instead of
$d\theta$ for the Lie algebra automorphism induced by $\theta$.

\begin{lemma}\label{max-cpt-conj}
Any two maximal lim-compact subgroups of $G$ are $\Aut(G)$--conjugate.
\end{lemma}

\begin{proof} Given two expressions $\varinjlim G_n = G =
\varinjlim G_n'$, corresponding to $\varinjlim V_n = V =
\varinjlim V_n$, we have an increasing function $f: \N \to \N$ such that
$V_n' \subset V_{f(n)}$.  Thus the two direct limit systems have a common
refinement, and we may assume $V_n' = V_n$ and $G_n' = G_n$.
It suffices now to show that the Cartan involutions
$\theta = \varinjlim \theta_n$ and
$\theta' = \varinjlim \theta_n'$ are conjugate in $\Aut(G)$.
\smallskip

Recursively we assume that $\theta_n$ and $\theta_n'$ are conjugate in
$\Aut(G_n)$, say $\theta_n' = \gamma_n\cdot\theta_n\cdot\gamma_n^{-1}$
for $n > 0$.  This gives an isomorphism between the direct systems
$\{(G_n,\theta_n)\}$ and $\{(G_n,\theta_n')\}$.  As in
\cite[Appendix A]{NRW1993} and \cite{W6.5} this results in
an automorphism of $G$ that conjugates $\theta$ to $\theta'$
in $\Aut(G)$ and sends $K$ to $K'$.
\end{proof}

The Lie algebra $\gg = \gk + \gs$ where $\gk$ is the
$(+1)$--eigenspace of $\theta$ and $\gs$ is the $(-1)$--eigenspace.
The Lie algebra $\gg = \gk + \gs$ where $\gk$ is the
$(+1)$--eigenspace of $\theta$ and $\gs$ is the $(-1)$--eigenspace.
We have just seen that any two choices of $K$ are
conjugate by an automorphism of $G$, so we have considerable freedom
in selecting $\gk$.  Also as  in the finite dimensional
case (and using the same proof), $[\gk,\gk] \subset \gk$,
$[\gk,\gs] \subset \gs$ and $[\gs,\gs] \subset \gk$.

\begin{proposition}\label{kgp}
Let $\gp$ be a flag-closed parabolic subalgebra of $\gg$\,, let $\theta$ be
a Cartan involution, and let $\gg = \gk + \gs$ be the corresponding
Cartan decomposition.  Then $\gg = \gk + \gp$.
\end{proposition}
\begin{proof}  If $\gk + \gp + \theta\gp \ne \gg$ then $\gg$
has nonzero elements $x \in (\gk + \gp + \theta\gp)^\perp$.  Any
such satisfies $x \perp \gn$, so $x \in \gp$, contradicting
$x \in (\gk + \gp + \theta\gp)^\perp$.  We have shown that
$\gg = \gk + \gp + \theta\gp$\,.
 
Let $x \in \gg$.  We want to show $x = 0$ modulo $\gk + \gp$.
Modulo $\gk$ we express $x = y + \theta z$ where $y, z \in \gp$\,.
Then $x - (y-z) = \theta z + z \in \gk$, so $x \in \gk$ modulo
$\gp$.  Now $x = 0$ modulo $\gk + \gp$.
\end{proof}

\begin{lemma} \label{pspan}
If $\gp$ is a flag-closed parabolic subalgebra of $\gg$\,,
and $\gp_{red,\R}$ is a reductive part, then $\gp_{red,\R}$ is
stable under some Cartan involution $\theta$ of $\gg$, and for
that choice of $\theta$ we have
$\gp = (\gp \cap \gk) + (\gp \cap \gs)+ \gn$.
\end{lemma} 

The global version of Proposition \ref{kgp} is the main result of this section:
\begin{theorem}\label{kgp-global}
Let $P$ be a flag-closed parabolic subgroup of $G$ and let $K$ be
a maximal lim--compact subgroup of $G$\,.  Then $G = K P$\,.
\end{theorem}

The proof of Theorem \ref{kgp-global} requires some riemannian geometry.
We collect a number of relevant semi--obvious (given the statement, the 
proof is obvious) results.  The key point here is that the real analytic
structure on $G$ defined in \cite{NRW1991} is the one for which
$\exp: \gg \to G$ restricts to a diffeomorphism of an open neighborhood
of $0 \in \gg$ onto an open neighborhood of $1 \in G$\,, and that
this induces a $G$--invariant analytic structure on $G/K$\,.

\begin{lemma}\label{geom-constr} Define $X = G/K$, with the
real analytic structure defined in \cite{NRW1991} and the $G$--invariant
riemannian metric defined by the positive definite $\Ad(K)$--invariant
bilinear form $\langle \xi, \eta \rangle = - \trace(\xi\cdot \theta \eta)$.  
Let $x_0 \in X$ denote the base point $1 K$\,. Then
$X$ is a riemannian symmetric space, direct limit of the finite dimensional
riemannian symmetric spaces $X_n = G_n(x_0) = G_n/K_n$,
and each $X_n$ is a totally geodesic submanifold of $X$.
\end{lemma}

The proof of Theorem \ref{kgp-global} will come down to an examination of
the boundary of $P(x_0)$ in $X$, and that will come down to an estimate 
based on

\begin{lemma} \label{normproj}
Let $\pi : \gg \to \gs$ be the 
$\langle \cdot, \cdot \rangle$--orthogonal projection, given by $\pi(\xi)
= \tfrac{1}{2}(\xi - \theta\xi)$.  If $\xi \in \gn$ then
$||\pi(\xi)||^2 = \tfrac{1}{2}||\xi ||^2$.  If $\gp$ is a
flag-closed parabolic then $\pi: (\gp \cap \gs)+ \gn \to \gs$
is a linear isomorphism, and if $\xi \in (\gp \cap \gs)+ \gn$
then $||\pi(\xi)||^2  \geqq \tfrac{1}{2}||\xi ||^2$.
\end{lemma}
\begin{proof} Whether $\gp$ is flag-closed or not, it is orthogonal
to $\gn$ relative to the trace form, so if $\xi \in \gn$ then
$\langle \xi , \theta\xi\rangle = -\trace(\xi \cdot\theta^2\xi) = 
-\trace(\xi\cdot\xi) = 0$.  Now $||\pi(\xi)||^2 = \tfrac{1}{4}(||\xi||^2 +
||\theta\xi||^2) = \tfrac{1}{2}||\xi ||^2$.
\smallskip

Now suppose that $\gp$ is flag-closed.  Then 
$\pi: (\gp \cap \gs)+ \gn \to \gs$
is a linear isomorphism by Lemma \ref{pspan}.  The summands
$\gp \cap \gs$ and $\gn$ are orthogonal relative to the trace form
so they are also orthogonal relative to $\langle\cdot,\cdot\rangle$
because $\langle \xi, \eta \rangle = -\trace (\xi\cdot \eta) = 0$ 
for $\xi \in \gn$ and $\eta \in \gp \cap \gs$.  
Note that their $\pi$--images are also orthogonal because
$\langle \pi(\theta\xi),\pi(\theta\eta)\rangle = \langle \pi(\theta\xi),\eta
\rangle$ vanishes using the opposite parabolic $\theta\gn + \gp_{red,\R}$.
Now $||\pi(\xi + \eta)||^2 = ||\pi(\xi)||^2 + ||\eta||^2 
\geqq \tfrac{1}{2}||\xi ||^2 + ||\eta||^2
\geqq \tfrac{1}{2}||\xi + \eta||^2$.
\end{proof}

Given $\eta \in \gs_R$, the riemannian distance 
dist$(x_0,\exp(\eta)x_0)$ from the base point $x_0$
to $\exp(\eta)x_0$ is $||\eta||$.  This can be seen directly, or it follows
by choosing $n$ such that $\eta \in \gg_n$ and looking in the symmetric
space $X_n$.  Now the second part of Lemma \ref{normproj} implies

\begin{lemma}\label{balls}
If $\gp$ is a flag-closed parabolic and $r > 0$ then the geodesic ball
$
B_X(r) = \{x \in X \mid \text{\rm dist}(x_0,x) < r\}
$
is contained in $\exp((\gp \cap \gs)+ \gn)x_0$. 
\end{lemma}

Finally we are in a position to prove the main result of this section.
\smallskip

{\bf Proof of Theorem \ref{kgp-global}.}
Let $\eta \in \gs_R$ with $||\eta || = 1$ and consider the geodesic 
$\gamma(t) = \exp(t\eta)x_0$ in $X$.  Here $t$ is arc length and $\gamma$
is defined on a maximal interval $a < t < b$ where $-\infty \leqq a< 0$
and $0 < b \leqq \infty$.  If $b < \infty$ choose $r > 0$ with $r < b$ and
$\xi \in (\gp \cap \gs)+ \gn$ such that $\exp(\xi)x_0 = \gamma(b-r)$.
Then $\gamma$ can be extended past $\gamma(b)$ inside the geodesic ball
$\exp(\xi)B_X(2r)$ of radius $2r$ with center $\exp(\xi)x_0$\,.  
That done, $t \mapsto \gamma(t)$ is defined on the interval $a < t < b+r$.
Thus $b = \infty$.  Similarly $a = -\infty$.  We have proved that if
$\gp$ is a flag-closed parabolic and $\eta \in \gs$ then
$\exp(t\eta)x_0 \in P(x_0)$ for every $t \in \R$.  In other words
$X = \exp(\gs)x_0$ is equal to $P(x_0)$.  That transitivity of
$P$ on $X = G / K$ is equivalent to $G = P K$\,.  
Under $x \mapsto x^{-1}$ that is the same as $G = K P$\,. 
\hfill $\square$

\section{Minimal Parabolic Subgroups}\label{sec4}
\setcounter{equation}{0}
In this section we study the subgroups of $G$ from which our
principal series representations are constructed.
\smallskip

\subsection*{4A. Structure.}
We specialize to the structure of minimal parabolic subgroups of the
classical real simple Lie groups $G$, extending structural results
from \cite{W7}.

\begin{proposition}\label{minlevi}
Let $\gp$ be a parabolic subalgebra of $\gg$ and $\gl$ a Levi
component of $\gp$.  If $\gp$ is a minimal parabolic subalgebra
then $\gl$ is a direct sum of
finite dimensional compact algebras $\gsu(p)$, $\gso(p)$ and $\gsp(p)$,
and their infinite dimensional
limits $\gsu(\infty)$, $\gso(\infty)$ and $\gsp(\infty)$.
If $\gl$ is a direct sum of
finite dimensional compact algebras $\gsu(p)$, $\gso(p)$ and $\gsp(p)$ and
their limits $\gsu(\infty)$, $\gso(\infty)$ and $\gsp(\infty)$,
then $\gp$ contains a minimal parabolic subalgebra of $\gg$ with
the same Levi component $\gl$.
\end{proposition}

\begin{proof} Suppose that $\gp$ is a minimal parabolic subalgebra
of $\gg$.  If a direct summand $\gl'$ of  $\gl$ has a
proper parabolic subalgebra $\gq$, we replace $\gl'$ by $\gq$
in $\gl$ and $\gp$.  In other words we refine the flag(s) that define
$\gp$.  The refined flag defines a parabolic $\gq \subsetneqq \gp$.
This contradicts minimality.  Thus no summand of $\gl$ has a proper
parabolic subalgebra.
Theorems \ref{glpar} and \ref{sosplevi} show that $\gsu(p)$, $\gso(p)$ and
$\gsp(p)$, and their limits $\gsu(\infty)$, $\gso(\infty)$ and $\gsp(\infty)$,
are the only possibilities for the simple summands of $\gl$.
\smallskip

Conversely suppose that the summands of $\gl$ are $\gsu(p)$, $\gso(p)$ and
$\gsp(p)$ or their limits $\gsu(\infty)$, $\gso(\infty)$ and $\gsp(\infty)$.
Let $(\cF, {'\cF})$ or $\cF$ be the flag(s) that define $\gp$.
In the discussion between Theorems \ref{glpar} and \ref{sosplevi} we described
a a minimal taut couple $(\cF_{min}, {'\cF}_{min})$ and a maximal taut couple
$(\cF_{max}, {'\cF}_{max})$ (in the $\gsl$ and $\ggl$ cases)
of semiclosed generalized flags which define parabolics that have the same
Levi component $\gl_\C$ as $\gp_\C$.  By construction $(\cF, {'\cF})$ refines
$(\cF_{min}, {'\cF}_{min})$ and $(\cF_{max}, {'\cF}_{max})$ refines
$(\cF, {'\cF})$.  As $(\cF_{min}, {'\cF}_{min})$ is uniquely defined from
$(\cF, {'\cF})$ it is $\tau$--stable.  Now the maximal $\tau$--stable
taut couple $(\cF^*_{max}, {'\cF}^*_{max})$ of semiclosed generalized flags
defines a $\tau$--stable parabolic $\gq_\C$ with the same Levi component $\gl_\C$
as $\gp_\C$, and $\gq := \gq_\C\cap\gg$ is a minimal parabolic
subalgebra of $\gg$ with Levi component $\gl$.
\smallskip

The argument is the same when $\gg_\C$ is $\gso$ or $\gsp$.
\end{proof}

Proposition \ref{minlevi} says that the Levi components of the minimal
parabolics are countable sums of compact real forms, in the sense of 
\cite{S}, of complex Lie algebras of types $\gsl$, $\gso$ and $\gsp$.  
On the group level, every element
of $M$ is elliptic, and $\gp_{red} = \gl \subsetplus \gt$ where
$\gt$ is toral, so every element of $\gp_{red}$ is semisimple.
This is where we use minimality of the parabolic $\gp$.
Thus $\gp_{red}\cap \gg_n$ is reductive in $\gg_{m,\R}$
for every $m \geqq n$.  Consequently we have Cartan involutions $\theta_n$ of
the groups $G_n$ such that $\theta_{n+1}|_{G_n} = \theta_n$
and  $\theta_n(M \cap G_n) = M \cap G_n$.
Now $\theta = \varinjlim \theta_n$
(in other words $\theta|_{G_n} = \theta_n$)
is a Cartan involution of $G$ whose fixed point set contains $M$.
We have just extended the argument of Lemma \ref{max-cpt-conj} to show that

\begin{lemma}\label{theta-aligned} $M$ is contained in a maximal 
lim-compact subgroup $K$ of $G$.
\end{lemma}

We fix a Cartan involution $\theta$ corresponding to the group
$K$ of Lemma \ref{theta-aligned}.

\begin{lemma}\label{construct-ma}
Decompose $\gp_{red} = \gm + \ga$ where
$\gm = \gp_{red}\cap \gk$ and $\ga = \gp_{red}\cap \gs$.
Then $\gm$ and $\ga$ are ideals in $\gp_{red}$ with $\ga$
commutative (in fact diagonalizable over $\R$).
In particular $\gp_{red} = \gm \oplus \ga$, direct sum of ideals.
\end{lemma}

\begin{proof}
Since $\gl = [\gp_{red}, \gp_{red}]$  we compute
$[\gm,\ga] \subset \gl \cap \ga = 0$.
In particular $[[\ga, \ga],\ga] = 0$.  So $[\ga, \ga]$
is a commutative ideal in the semisimple algebra $\gl$, in other words
$\ga$ is commutative.
\end{proof}

The main result of this subsection is the following generalization of the
standard decomposition of a finite dimensional real parabolic.  We have
formulated it to emphasize the parallel with the finite dimensional case.
However some details of the construction are rather different; see 
Proposition \ref{construct-p} and the discussion leading up to it.

\begin{theorem}\label{lang-alg}
The minimal parabolic subalgebra $\gp$ of $\gg$ decomposes as
$\gp = \gm + \ga + \gn =
\gn \subsetplus (\gm \oplus \ga)$, where $\ga$ is commutative,
the Levi component $\gl$ is an ideal in $\gm$\,, and $\gn$ is 
the linear nilradical $\gp_{nil}$.  On the group level,
$P = M A N = N \ltimes (M \times A)$ where
$N = \exp(\gn)$ is the linear unipotent radical of $P$, 
$A = \exp(\ga)$ is diagonalizable over $\R$ and isomorphic to a 
vector group, and 
$M = P \cap K$ is limit--compact with Lie algebra $\gm$\, .
\end{theorem}

\begin{proof}
The algebra level statements come out of Lemma \ref{construct-ma} and the
semidirect sum decomposition $\gp = \gp_{nil} \subsetplus \gp_{red}$.
\smallskip

For the group level statements, we need only check that $K$ meets every
topological component of $P$.  Even though $P\cap G_n$ need not
be parabolic in $G_n$, the group $P\cap\theta P\cap G_n$ is
reductive in $G_n$ and $\theta_n$--stable, so $K_n$ meets each of
its components.  Now $K$ meets every component of $P\cap\theta P$.
The linear unipotent radical of $P$ has Lie algebra $\gn$ and thus
must be equal to $\exp(\gn)$, so it does not effect components.  Thus
every component of $P_{red}$ is represented by an element of 
$K \cap P\cap\theta P = K \cap P = M$.  That derives
$P = M A N = N \ltimes (M \times A)$ from
$\gp = \gm + \ga + \gn = \gn \subsetplus (\gm \oplus \ga)$.
\end{proof}

\subsection*{4B. Construction.}
Given a subalgebra $\gl \subset \gg$ that is the Levi component of
a minimal parabolic subalgebra $\gp$\,, we will extend the notion of
{\em standard} of Definition \ref{standard} from simple ideals of $\gl$
to minimal parabolics and their reductive parts.  The construction of the
standard flag-closed minimal parabolic 
$\gp^\dagger = \gm + \ga^\dagger + \gn^\dagger$
with the same Levi component as $\gp = \gm + \ga + \gn$
will tell us that $K$ is transitive on $G/P^\dagger$, and this
will play an important role in construction of Harish--Chandra modules of
principal series representations.
\smallskip

We carry out the construction in detail for the cases where $\gg$ 
is defined by a hermitian 
form $f: V_\F \times V_\F \to \F$\,, where $\F$ is $\R$, $\C$ 
or $\H$.  The idea is the same for the other cases.
See Proposition \ref{construct-p} below.
\smallskip

Write $V_\F$ for $V$ as a real, complex or quaternionic vector space,
as appropriate, 
and similarly for $W_\F$.  We use $f$ for an 
$\F$--conjugate--linear identification of $V_\F$ and $W_\F$.
We are dealing with the Levi component $\gl =
\bigoplus_{j \in J}\, \gl_{j,\R}$ of a minimal self--normalizing parabolic
$\gp$, where the $\gl_{j,\R}$ are simple and
standard in the sense of Definition \ref{standard}.  
Let $X_\F^{levi}$ denote the sum of the
corresponding subspaces $(X_j)_\F \subset V_\F$ and $Y_\F^{levi}$ the
analogous sum of the $(Y_j)_\F \subset W_\F$.  Then $X_\F$ and $Y_\F$
are nondegenerately paired.  Of course they may be small, even zero.
In any case, 
\begin{equation}\label{fill-out-VW}
\begin{aligned}
&V_\F = X_\F^{levi} \oplus (Y_\F^{levi})^\perp \, ,
W_\F = Y_\F^{levi} \oplus (X_\F^{levi})^\perp, \text{ and } \\
&(X_\F^{levi})^\perp \text{ and } (Y_\F^{levi})^\perp
\text{ are nondegenerately paired.}
\end{aligned}
\end{equation}
These direct sum decompositions (\ref{fill-out-VW}) now become
\begin{equation}\label{fill-out-V}
V_\F = X_\F^{levi} \oplus (X_\F^{levi})^\perp \quad \text{ and } \quad f 
\text{ is nondegenerate on each summand.} 
\end{equation}
Let $X'$ and $X''$ be paired maximal isotropic subspaces of 
$(X_\F^{levi})^\perp$.  Then 
\begin{equation}\label{expand-V}
V_\F = X_\F^{levi} \oplus (X'_\F \oplus X''_\F) \oplus Q_\F \text{ where }
Q_\F := (X_\F^{levi} \oplus (X'_\F \oplus X''_\F))^\perp .
\end{equation}

The subalgebra 
$\{\xi \in \gg \mid \xi(X_\F \oplus Q_\F) = 0\}$
of $\gg$ has maximal toral subalgebras contained in $\gs$, 
in which every element has all eigenvalues real.  The one we will use is
\begin{equation}\label{def-a'}
\begin{aligned}
\ga^\dagger = {\bigoplus}_{\ell \in C}\, &\ggl(x_\ell'\R,x_\ell''\R) 
        \text{ where }\\
        & \{x'_\ell \mid \ell \in C\} \text{ is a basis of } X'_\F 
        \text{ and } \\
        & \{x''_\ell \mid \ell \in C\} \text{ is the dual basis of } X''_\F.
\end{aligned}
\end{equation}
It depends on the choice of basis of $X'_\F$\,.
Note that $\ga^\dagger$ is abelian, in fact diagonal over $\R$ as defined.
\smallskip

As noted in another argument, in the discussion between Theorems 
\ref{glpar} and \ref{sosplevi} we described
a minimal taut couple $(\cF_{min}, {'\cF}_{min})$ and a maximal taut couple
$(\cF_{max}, {'\cF}_{max})$ (in the $\gsl$ and $\ggl$ cases)
of semiclosed generalized flags which define parabolics that have the same
Levi component $\gl_\C$ as $\gp_\C$.   That argument of \cite{DC3} does not 
require simplicity of the $\gl_j$.  It works with 
$\{\gl_j\}_{j \in J} \cup \{\ggl(x_\ell'\R,x_\ell''\R)\}_{\ell \in C}$
and a total ordering on $J^\dagger := J \dot\cup C$ that restricts to the
given total ordering on $J$.  Any such interpolation of the index $C$ of 
(\ref{def-a'}) into the totally ordered index set $J$ of 
$X_\F^{levi} = \bigoplus_{j \in J} (X_j)_\F$ 
(and usually there will be infinitely many) gives a self--taut semiclosed
generalized flag $\cF^\dagger$ and defines a minimal self--normalizing 
parabolic subalgebra $\gp^\dagger$ of $\gg$ with the same Levi component 
as $\gp$ The decompositions corresponding to (\ref{fill-out-VW}),
(\ref{fill-out-V}) and (\ref{expand-V}) are given by 
\begin{equation} \label{dagger-sum}
X_\F^\dagger = {\bigoplus}_{d \in J^\dagger} (X_d)_\F = 
X_\F^{levi} \oplus (X'_\F \oplus X''_\F) \text{ and } 
Q_\F^\dagger = Q_\F.
\end{equation}

In the discussion just above, $\gp^\dagger$ is the stabilizer of the
flag $\cF^\dagger$.  The nilradical of $\gp^\dagger$ is defined
by $\xi X_d \subset \bigoplus_{d' < d} X_{d'}$ and $\xi Q^\dagger_\F = 0$.
\smallskip

In addition, the subalgebra
$\{\xi \in \gp \mid \xi(X_\F^{levi}  \oplus (X'_\F \oplus X''_\F)) = 0\}$
has a maximal toral subalgebra $\gt'$ in which every eigenvalue is pure 
imaginary, because $f$ is definite on $Q_\F$.  It is
unique because it has derived algebra zero and is given by the action of the 
$\gp$--stabilizer of $Q_\F$ on the definite subspace $Q_\F$.  
This uniqueness tell us that $\gt'$ is the same for $\gp$ and 
$\gp^\dagger$.
\smallskip

Let $\gt''$ denote the maximal toral subalgebra in
$\{\xi \in \gp \mid \xi(X_\F \oplus Q_\F) = 0\}$.  It stabilizes each 
Span($x'_\ell ,x''_\ell$) in (\ref{def-a'}) and centralizes $\ga^\dagger$,
so it vanishes if $\F \ne \C$.  The $\gp^\dagger$ analog of $\gt''$
is $0$ because $X^\dagger_\F \oplus Q_\F = V_\F$.  In any case we have
\begin{equation}\label{def-t}
\gt = \gt^\dagger := \gt' \oplus \gt''\,.
\end{equation}

For each $j \in J$ we define
an algebra that contains $\gl_{j,\R}$ and acts on $(X_j)_\F$ by:
if $\gl_{j,\R} = \gsu(*)$ then $\widetilde{\gl_{j,\R}} = \gu(*)$ (acting
        on $(X_j)_\C$); otherwise $\widetilde{\gl_{j,\R}} = \gl_{j,\R}$.
Define 
\begin{equation}\label{def-m'}
\widetilde{\gl} = 
  \bigoplus_{j \in J}\, \widetilde{\gl_{j,\R}} \quad
\text{ and } \quad \gm^\dagger = \widetilde{\gl} + \gt\, .
\end{equation}
Then, by construction, $\gm^\dagger = \gm$.  Thus 
$\gp^\dagger$ satisfies
\begin{equation}\label{def-p'}
\gp^\dagger := \gm +\ga^\dagger + \gn^\dagger=
        \gn^\dagger \subsetplus (\gm \oplus \ga^\dagger).
\end{equation}
Let $\gz$ denote the centralizer of $\gm \oplus \ga$ in $\gg$
and let $\gz^\dagger$ denote the centralizer of 
$\gm \oplus \ga^\dagger$ in $\gg$.  We claim
\begin{equation}\label{red-parts}
\gm + \ga = \widetilde{\gl} + \gz \text{ and }
        \gm + \ga^\dagger = \widetilde{\gl} + \gz^\dagger
\end{equation}
For by construction $\gm \oplus \ga =
\widetilde{\gl} + \gt + \ga \subset \widetilde{\gl} + \gz$.
Conversely if $\xi \in \gz$ it preserves each $X_{j,\F}$, each joint
eigenspace of $\ga$ on $X'_\F \oplus X''_\F$, and each joint
eigenspace of $\gt$, so $\xi \subset \widetilde{\gl} + \gt + \ga$.
Thus $\gm + \ga = \widetilde{\gl} + \gz$.  The same argument
shows that $\gm + \ga^\dagger = \widetilde{\gl} + \gz^\dagger$.
\smallskip

If $\ga$ is diagonalizable as in the definition (\ref{def-a'}) of 
$\ga^\dagger$, in other words if it is a sum of standard $\ggl(1;\R)$'s,
then we could choose $\ga^\dagger = \ga$, hence could
construct $\cF^\dagger$ equal to $\cF$, resulting in $\gp = \gp^\dagger$.
In summary:

\begin{proposition}\label{construct-p}
Let $\gg$ be defined by a hermitian form and let $\gp$ be a minimal
self--normalizing parabolic subalgebra. In the notation above, 
the standard parabolic $\gp^\dagger$
is a minimal self--normalizing parabolic subalgebra of $\gg$ with 
$\gm^\dagger = \gm$.  In particular $\gp^\dagger$ and $\gp$ 
have the same Levi component.  Further we can take $\gp^\dagger = \gp$ 
if and only if $\ga$ is the sum of commuting standard $\ggl(1;\R)$'s.
\end{proposition}

Similar arguments give the construction behind Proposition \ref{construct-p} 
for the other real simple direct limit Lie algebras.
\smallskip

Note also from the construction of $\gp^\dagger$ we have
\begin{proposition}\label{transitive}
The standard parabolic $\gp^\dagger$ constructed above, is 
flag-closed.  In particular, by {\rm Theorem \ref{kgp-global}}, the
maximal lim-compact subgroup $K$ of $G$ is transitive on
$G / P^\dagger$\,, and so $G / P^\dagger \cong K / M^\dagger$
as real analytic manifolds.
\end{proposition}

$P$ and $P^\dagger$ are minimal self normalizing parabolic subgroups
of $G$.  We will discuss representations of $P$ and $P^\dagger$,
and the induced representations of $G$.  The latter are the {\em principal 
series} representations of $G$ associated to $\gp$ and $\gp^\dagger$, 
or more precisely to
the pair $(\gl,J)$ where $\gl$ is the Levi component and $J$ is
the ordered index set for the simple summands of $\gl$.

\section{Amenable Induction}\label{sec5}
\setcounter{equation}{0}

In this section we study amenable groups and invariant means in the
context of quotients $G / P$ by minimal parabolic subgroups.
This allows us to construct induced representations without local
compactness or invariant measures.
\medskip

\subsection*{5A. Amenable Groups.}
We consider a topological group $G$ which is {\em not}
assumed to be locally compact, and a closed subgroup $H$ of $G$.  
We follow D. Belti\c{t}\u{a} \cite[Section 3]{Bel} for amenability 
on $H$.  Consider the commutative $C^*$ algebra 
$$
L^\infty(G/H) = \{f : G/H \to \C \text{ continuous } \mid
{\sup}_{x \in G/H} |f(x)| < \infty\}.
$$
It has pointwise multiplication, norm $||f|| = \sup_{x \in G/H} |f(x)|$
and unit given by $\mathbf 1 (x) = 1$.  
We denote the usual left and right actions of $G$ on $L^\infty(G)$ by
$(\ell(g)f)(k) = f(g^{-1}k)$ and $(r(g)f)(k) = f(kg)$.  We often
identify $L^\infty(G/H)$ with the closed subalgebra of $L^\infty(G)$
consisting of $r(H)$--invariant functions.
\smallskip

The space of {\em right uniformly continuous bounded functions} on 
$G/H$ is
\begin{equation}\label{def-rucb}
RUC_b(G/H) = \{f \in L^\infty(G/H) \mid x \mapsto \ell(x)f
	\text{ continuous } G \to L^\infty(G/H)\}.
\end{equation}
In other words,
\begin{equation}\label{eps}
\text{if } \epsilon > 0,\, \exists \text{ nbhd }
U \text{ of } 1 
\text{ in } G \text{ s.t. } |f(ux) - f(x)| < \epsilon \text{ for } 
x \in G/H, u \in U.
\end{equation}
Similarly, the space $LUC_b(G)$ of {\em left uniformly continuous bounded 
functions} on $G$ is
$
\{f \in L^\infty(G) \mid x \mapsto r(x)f
        \text{ is a continuous map } G \to L^\infty(G)\}.
$
\begin{lemma}\label{l-cont}
The left action of $G$ on $RUC_b(G/H)$ is a 
continuous representation.
\end{lemma}
\begin{proof} (\ref{def-rucb}) and (\ref{eps}) give
$||\ell(u)f - f'||_\infty \leqq ||\ell(u)f - f||_\infty + ||f - f'||_\infty$.
\end{proof}

\begin{example}\label{lub-coef}{\rm
Let $\varphi$ be a unitary representation of $G$.  This means a weakly
continuous homomorphism into the unitary operators on a separable Hilbert
space $E_\varphi$\,. If $u, v \in E_\varphi$ the {\em coefficient} function
$f_{u,v}: G \to \C$ is $f_{u,v}(x) = \langle u, \varphi(x)v\rangle$.  Let
$\epsilon > 0$ and choose a neighborhood $B$ of $1$ in $G$ such that
$||u||\cdot ||v - \varphi(y)v|| < \epsilon$ for $y \in B$.
Then $|f_{u,v}(x) - f_{u,v}(xy)| < \epsilon$ for all $x \in G$ and $y \in B$,
so $f_{u,v} \in LUC_b(G)$.  Similarly, choose a neighborhood $B'$ of $1$
such that $||u - \varphi(y)u||\cdot ||v|| < \epsilon$ for $y \in B'$.
Then $|f_{u,v}(x) - f_{u,v}(y^{-1}x)| < \epsilon$ for all $x \in G$ 
and $y \in B'$, so $f_{u,v} \in RUC_b(G)$.
}\hfill $\diamondsuit$
\end{example}
A {\em mean} on $G/H$ is a linear functional
$\mu: RUC_b(G/H) \to \C$ such that
\begin{equation}\label{l-inv-mean}
\begin{aligned}
\text{ (i)} &\phantom{X}\mu(\mathbf 1) = \mathbf 1 \text{ and } \\
\text{ (ii)} &\phantom{X}\text{if } f(x) \geqq 0 \text{ for all } x \in G/H
	\text{ then } \mu(f) \geqq 0.
\end{aligned}
\end{equation}

Any left invariant mean $\mu$ on $G/H$ is a continuous functional on 
$RUC_b(G/H)$ and satisfies $||\mu|| = 1$.  
\smallskip

The topological group $H$ is {\em amenable} if it has a left invariant 
mean, or equivalently (using $h \mapsto h^{-1}$) if it has a right invariant 
mean.

\begin{proposition}\label{k-amenable}{\rm (See} {\sc (Belti\c{t}\u{a} 
\cite[Example 3.4]{Bel})} Let $\{H_\alpha\}$ be a strict direct system of 
amenable topological groups.  Let $H$ be a topological group in which the 
algebraic direct limit $\varinjlim H_\alpha$ is dense.  Then $H$ is amenable.
\end{proposition}

When we specialize this to our Lie group setting it will be useful to denote
\begin{equation}\label{invmean}
\cM(G/H): \text{ all means on } G/H \text{ with the action } (\ell(g)\mu)(f)
	= \mu(\ell(g^{-1})f).
\end{equation}

\begin{lemma}
\label{r-inv-mean}
Let $G$ be a topological group and $H$ a closed amenable subgroup.
Then $\cM(G/H) \ne \emptyset$.
\end{lemma}
Lemma \ref{r-inv-mean} is a refinement, suggested by G. \' Olafsson, 
to my original argument.
We need it for the sharpening \cite{OW} of the principal series construction
of \S 5C below.

\begin{proof} Let $f_1 \in RUC_b(G/H)$ not identically zero and with all
values $\geqq 0$.  Taking a left $G$--translate and then
scaling, we may assume $f_1(1H) = 1 = ||f_1||_\infty$\,.  Now view $f_1$ as
an $r(H)$--invariant function on $G$.  Let $\mu$ be a right invariant mean
on $H$.  Then $f \mapsto \mu(f|_H)$ defines a right $H$--invariant
mean $\widetilde{\mu}$ on $G$, in other words a mean on $G/H$, and
$\widetilde{\mu}(f_1) > 0$.
\end{proof}

A similar argument gives the following, which is well known in the locally
compact case and probably known in general:

\begin{lemma} \label{ext-mean}
If $H_1$ is a closed normal amenable subgroup of $H$ and $H/H_1$ is amenable
then $H$ is amenable.
\end{lemma}
\begin{proof} 
Let $\mu$ be a left invariant mean on $H_1$ and $\nu$ a left invariant mean
on $H/H_1$.  Given $f \in RUC_b(H)$ and $h \in H$ define 
$f_h = (\ell(h^{-1})(f))|_{H_1} \in RUC_b(H_1)$, so
$f_h(y) = f(hy)$ for $y \in H_1$.  If
$y' \in yH_1$ then $\mu(f_{y'}) = \mu(\ell(y'^{-1}y)f_y) = \mu(f_y)$, so
we have $g_f \in RUC_b(H/H_1)$ defined by $g_f(hH_1) = \mu(f_h)$.  That defines
a mean $\beta$ on $G$ by $\beta(f) = \nu(g_f)$, and $\beta$ is left
invariant because $\beta(\ell(a)f) = \nu(g_{\ell(a)f}) =
\nu(\ell(a^{-1})g_{\ell(a)f}) = \beta(f)$.
\end{proof}

\begin{theorem}\label{kp-amenable} The maximal lim--compact subgroups
$K = \varinjlim K_n$ of $G$ are amenable.  Further,
the minimal parabolic subgroups of $G$ are amenable.  Finally, if
$P$ is a minimal parabolic subgroup of $G$ then $\cM(G/P) \ne \emptyset$.
\end{theorem}
In \cite{OW} we will see more: that $\cM(G/P)$ separates points on $RUC_b(G/P)$.
\begin{proof} By construction $K$ is a direct limit of compact 
(thus amenable) groups, so it is amenable by Proposition \ref{k-amenable}.
In Theorem \ref{lang-alg} we saw the decomposition 
$P = M A N$ of the minimal parabolic subgroup.  $M$ is
amenable because it is a closed subgroup of the amenable group $K$.
$A N$ is a direct limit of finite dimensional connected solvable 
Lie groups, hence is amenable.  And now the semidirect product 
$P = (A N)\rtimes M$ is amenable by Lemma \ref{ext-mean}.
Finally, Lemma \ref{r-inv-mean} says that $\cM(G/P) \ne \emptyset$.
\end{proof}

\subsection*{5B. Induced Representations: General Construction.} Here is the 
general construction for amenable induction.  
Let $G$ be a topological group and $H$ a closed amenable subgroup.
A unitary
representation $\tau \in \widehat{H}$, say with representation space
$E_\tau$, defines an $G$--homogeneous Hilbert space bundle
$\E_\tau \to G/H$.  Using the set $\cM(G/H)$
of Theorem \ref{kp-amenable}, we are
going to define an induced representations $\Ind_H^G(\tau)$ of $G$.
The representation space will be a complete locally convex topological 
vector space.
\medskip

Denote the space of bounded, right uniformly continuous sections
of $\E_\tau \to G/H$ by $RUC_b(G/H;\E_\tau)$.  Given
$\omega \in RUC_b(G/H;\E_\tau)$ we have the pointwise norm
$||\omega||: G/H \to \C$.  
Note that $||\omega|| \in RUC_b(G/H)$.  
Now each mean $\mu \in \cM(G/H)$ defines a global seminorm
$$
\nu_\mu(\omega) = \mu(||\omega||)
$$
on $RUC_b(G/H;\E_\tau)$.  Given any left $G$--invariant subset
$\cM'$ of $\cM(G/H)$ we define
\begin{equation}\label{kernel}
J_{\cM'}(G/H;\E_\tau) = \{\omega \in RUC_b(G/H;\E_\tau) \mid 
	\nu_\mu(\omega) = 0 \text{ for all } \mu \in \cM' \}.
\end{equation}
The seminorms $\nu_{\mu}$, $\mu \in \cM'$, descend to 
$RUC_b(G/H;\E_\tau)/J_{\cM'}(G/H;\E_\tau)$.  That family of seminorms
defines the complete locally convex topological vector space
\begin{equation}\label{lc-space}
\Gamma_{\cM'}(G/H;\E_\tau): \text{ completion of }
        \tfrac{RUC_b(G/H;\E_\tau)}{J_{\cM'}(G/H;\E_\tau)} \text{ relative to }
        \{\nu_{\mu} \mid \mu \in \cM'\}.
\end{equation}

\begin{proposition} \label{gen-construction}
The natural action of $G$ on the complete locally convex
topological vector space $\Gamma_{\cM'}(G/H;\E_\tau)$ is a 
continuous representation of $G$.
\end{proposition}
Making use of the result \cite[Proposition 1]{OW} of G. \' Olafsson
and myself, which says that $J_{\cM(G/H)}(G/H;\E_\tau) = 0$, and writing
$$
\Gamma(G/H;\E_\tau) := \Gamma_{\cM(G/H)}(G/H;\E_\tau),
$$
we have the special case
\begin{corollary} \label{sep-gen-construction}
The natural action of $G$ on the complete locally convex
topological vector space $\Gamma(G/H;\E_\tau)$ is a
continuous representation of $G$.
\end{corollary}

\subsection*{5C. Principal Series Representations.}  
We specialize the construction
of Proposition \ref{gen-construction} to our setting where $G$ is a real
Lie group with complexification $GL(\infty;\C)$, $SL(\infty;\C)$,
$SO(\infty;\C)$ or $Sp(\infty;\C)$, and where $P$ is a minimal 
self--normalizing parabolic subgroup.  Theorem \ref{kp-amenable} ensures 
that $\cM(G/P)$ is non--empty, and \cite[Proposition 1]{OW} says that
it separates elements of $RUC_b(G/P;\E_\tau)$.  Given a unitary 
representation $\tau$ of $P$ we then have 
\begin{itemize}
\item the $G$--homogeneous hermitian vector bundle $\E_\tau \to  G/P$,
\item the seminorms $\nu_{\mu}$, $\mu \in \cM(G/P;\E_\tau)$, on 
	$RUC_b(G/P;\E_\tau)$, and
\item the completion $\Gamma(G/P;\E_\tau)$ of $RUC_b(G/P;\E_\tau)$
	relative to that collection of seminorms, which is a
	complete locally convex topological vector space.
\end{itemize}
\begin{definition}\label{def-ind}{\rm
The representation $\pi_\tau$ of $G$ on $\Gamma(X;\E_\tau)$ is 
{\em amenably induced} from $(P,\tau)$ to $G$.  We denote
it $\Ind_P^G(\tau)$.  The family of all such 
representations forms the {\em general principal series} of
representations of $G$.} \hfill $\diamondsuit$
\end{definition}

\begin{proposition}\label{k-spec-1}
If the minimal self--normalizing parabolic $P$ is flag-closed,
and $\tau$ is a unitary representation of $P$, then 
$\Ind_P^G(\tau)|_K = \Ind_M^K(\tau|_M)$.
\end{proposition}
Proposition \ref{k-spec-1} makes indirect use of \cite{OW} to shorten
and simplify my original argument.
\begin{proof}
Since $P$ is flag closed, Theorem \ref{kgp-global} says that 
$K$ is transitive on $X = G/P$, so $X = K/M$ as well.  Thus 
$\E_\tau \to X$ can be viewed 
as the $K$--homogeneous Hilbert space bundle $\E_{\tau|_K} \to X$ 
defined by $\tau|_K$\,.  Evidently $RUC_b(X;\E_\tau) = RUC_b(X;\E_{\tau|_K})$.
Now we have a $K$--equivariant identification
$\cM(K/M;\E_{\tau|_K}) = \cM(G/P;\E_\tau)$, resulting in a $K$--equivariant
isomorphism of $\Gamma(K/M;\E_{\tau|_K})$ onto $\Gamma(G/P;\E_\tau)$,
which in turn gives a topological equivalence of $\Ind_M^K(\tau|_M)$ with
$\Ind_P^G(\tau)|_K$.
\end{proof}

In the current state of the art, this construction provides more questions
than answers.  Some of the obvious questions are
\begin{itemize}
\item[1.] When does $\Gamma(X;\E_\tau)$ have a $G$--invariant Fr\' echet space 
	structure?  When it exists, is it nuclear?
\item[2.] When does $\Gamma(X;\E_\tau)$ have a $G$--invariant Hilbert space
        structure?  In other words, when is $\Ind_P^G(\tau)$ unitarizable?
\item[3.] What is the precise $K$--spectrum of $\pi_\tau$?
\item[4.] When is the space of smooth vectors dense in $\Gamma(X;\E_\tau)$?  
	In other words, when (or to what extent) 
	does the universal enveloping algebra $\cU(\gg)$ act?
\item[5.] If $\tau|_M$ is a factor representation of type $II_1$, and $P$ is
	flag closed, does the character of $\tau|_M$ lead to an analog of
	character for $\Ind_P^G(\tau)$, or for $\Ind_M^K(\tau|_M)$?
\end{itemize}
The answers to (1.) and (2.) are well known in the finite dimensional case.
They are also settled (\cite{W5}) when $G = \varinjlim G_n$ restricts to $P =
\varinjlim P_n$ with $P_n$ minimal parabolic in $G_n$.  However that is a
very special situation.  The answer to (3.) is only known in special
finite dimensional situations.  Again, (4.) is classical in the finite
dimensional case, and also clear in the cases studied in \cite{W5}, but in 
general one expects that the answer will depend on better
understanding of the possibilities for $\tau$ and the structure of
$\cM(G/P)$.  For that we append to this paper a short discussion of unitary
representations of self normalizing minimal parabolic subgroups.

\section*{Appendix: Unitary Representations of Minimal Parabolics.}
\setcounter{equation}{0}
\addtocounter{section}{1}

In order to describe the unitary representations $\tau$ of $P$ that are basic
to the construction of the principal series in Section \ref{sec5}, we must 
first choose a class of representations.  The best choice is not clear, so
we indicate some of the simplest choices.
\smallskip

\subsection*{Reductions.}
First, we limit complications by looking only at unitary representations $\tau$
of $P = MAN$ that annihilate the linear nilradical $N$.  Since the structure
of $N$ is not explicit, especially since we do not necessarily have a 
restricted root decomposition of $\gn$, the unitary representation theory
of $N$ and the corresponding extension with representations of $MA$ present
serious difficulties, which we will avoid.  This is in accord with the 
finite dimensional setting.
\smallskip

Second, we limit
surprises by assuming that $\tau|_A$ is a unitary character.  This too is in
accord with the finite dimensional setting.  Thus we are looking at
representations of the form $\tau(man)v = e^{i\lambda(\log a)}\tau(m)v$,
$v \in E_\tau$, 
where $\lambda \in \ga^*$ is a linear functional on $\ga$ and $\tau|_M$
is a unitary representation of $M$.  
\smallskip

We know the structure of $\gl$
from Proposition \ref{minlevi}, and the construction of $\gm$ from $\gl$
from (\ref{chev}) and Lemma \ref{construct-ma}.  Thus
we are then in a position to take advantage of known results on unitary
representations of lim-compact groups to obtain the factor representations
of the identity component $M^0$.  Lemma \ref{2-group} below, shows how the
unitary representations of $M$ are constructed from the unitary representations
of $M^0$.

\begin{lemma}\label{2-group}
$M = M^0 \times (A_\C \cap K)$ and every element of $A_\C \cap K$ has
square $1$.  In other words, $M$ is the direct product of its identity
component with a direct limit of elementary abelian $2$--groups.
\end{lemma}
\begin{proof} The parabolic $P_\C$ is self--normalizing, and self-normalizing
complex parabolics are connected.  Thus $M_\C$ and $A_\C$ are connected. 
Now $M_\C \cap G$ is connected, and the topological components of $M$ are
given by $A_\C \cap K$.  If $x \in A_\C \cap K$ then $x = \theta x = x^{-1}$.
\end{proof}

Third, we further limit surprises by assuming that $\tau|_{A_\C \cap K}$ is
a unitary character $\chi$.  In other words, there is a unitary character 
$e^{i\lambda} \otimes \chi$ on $(A_\C \cap G) = A \times (A_\C \cap K)$
such that $\tau(m_0 m_a an)v = e^{i\lambda(\log a)} \chi(m_a)\tau(m_0)v$
for $m_0 \in M^0$, $m_a \in A_\C \cap K$, $a \in A$ and $n \in N$.
\smallskip

Using (\ref{chev}) and Lemma \ref{construct-ma} we have 
$\gm = \gl \subsetplus \gt$ and $[\gm,\gm] = \gl$ where $\gt$ is toral.
So $M^0$ is the semidirect product $L \rtimes T$ where $T$ is a direct
limit of finite dimensional torus groups.  Let $\widetilde{L}$ be the
group obtained from $L$ by replacing each special unitary factor $SU(*)$ 
by the slightly larger unitary group $U(*)$.  This absorbs a factor from
$T$ and the result is a direct product decomposition
\begin{equation}\label{m-chev}
M^0 = \widetilde{L} \times \widetilde{T} \text{ where } \widetilde{T} 
	\text{ is toral.}
\end{equation}
Our fourth restriction, similar to the second and third, is that
$\tau|_{\widetilde{T}}$ be a unitary character.
\smallskip

In summary, we are looking at unitary representations $\tau$ of $P$ whose 
kernel contains $N$ and which restrict to unitary characters on the
commutative groups $A$, $A_\C \cap K$ and $\widetilde{T}$.  Those unitary
characters, together with the unitary representation $\tau|_{\widetilde{L}}$,
determine $\tau$.

\subsection*{Representations.}  We discuss some possibilities for an
appropriate class $\cC(\widetilde{L})$ of representations of $\widetilde{L}$.
The standard group $\widetilde{L}$ is a product of standard groups $U(*)$, 
and possibly one factor $SO(*)$ or $Sp(*)$.  The representation theory of the
finite dimensional groups $U(n)$, $SO(n)$ and $Sp(n)$ is completely 
understood, so we need only consider the cases of $U(\infty)$,
$SO(\infty)$ and $Sp(\infty)$.  
We will indicate some possibilities for $\cC(U(\infty))$.  The situation
is essentially the same for $SO(\infty)$ and $Sp(\infty)$.
\smallskip

\underline{Tensor Representations of $U(\infty)$}.  In the classical setting,
the symmetric group $\gS_n$ permutes factors of $\bigotimes^n(\C^p)$. The
resulting representation of $U(p) \times \gS_n$ specifies representations
of $U(p)$ on the various irreducible summands for that action of $\gS_n$.  
These summands occur with multiplicity $1$.  See Weyl's book \cite{W}.  
Segal \cite{Se}, Kirillov \cite{K}, and Str\u atil\u a \& Voiculescu
\cite{SV}  developed and proved an analog of this for $U(\infty)$.  
These ``tensor representations'' are factor representations of type 
$II_\infty$, but they do not extend by continuity to the class of unitary
operators of the form $identity + compact$.
See \cite[Section 2]{SV2} for a treatment of this topic.
Because of this limitation one should also consider other classes of 
factor representations of $U(\infty)$.
\smallskip
 
\underline{Type $II_1$ Representations of $U(\infty)$}.
If $\pi$ is a continuous unitary finite factor representation of $U(\infty)$,
ten it has a well defined character $\chi_\pi(x) = \trace \pi(x)$, the
normalized trace.
Voiculescu \cite{V} worked out the parameter space for these finite
factor representations.  It consists of all bilateral sequences
$\{c_n\}_{-\infty < n < \infty}$ such that 
(i) $\det((c_{m_i + j - i})_{1 \leqq i, j \leqq N} \geqq 0$ for
          $m_i \in \Z$ and $N \geqq 0$ and 
(ii) $\sum c_n = 1$.
The character corresponding to $\{c_n\}$ and $\pi$ is 
$\chi_\pi(x) = \prod_i p(z_i)$ where $\{z_i\}$ is the multiset of
eigenvalues of $x$ and 
$p(z) = \sum c_nz^n$. Here $\pi$ extends to the group of all unitary operators 
$X$ on the Hilbert space completion of $\C^\infty$ such that $X - 1$ is
of trace class.  See \cite[Section 3]{SV2} for a more detailed summary.
This is a very convenient choice of class $\cC_{U(\infty)}$, and it
is closely tied to
the Olshanskii--Vershik notion (see \cite{O}) of tame representation.
\smallskip

\underline{Other Factor Representations of $U(\infty)$}.
Let $\cH$ be the Hilbert space completion of $\varinjlim \cH_n$ where
$\cH_n$ is the natural representation space of $U(n)$.  Fix a bounded
hermitian operator $B$ on $\cH$ with $0 \leqq B \leqq I$.  Then
$$
\psi_B : U(\infty) \to \C\,, \text{ defined by }
\psi_B(x) = \det((1 - B) + B x)
$$
is a continuous  function of positive type on $U(\infty)$.  Let $\pi_B$
denote the associated cyclic representation of $U(\infty)$.  Then
(\cite[Theorem 3.1]{SV3}, or see \cite[Theorem 7.2]{SV2}),
\begin{itemize}
\item[(1)] $\,\,\psi_B$ is of type $I$ if and only if $B(I-B)$ is of trace class.
	In that case $\pi_B$ is a direct sum of irreducible representations.
\item[(2)] $\,\,\psi_B$ is factorial and type $I$ if and only if $B$ is a 
	projection.  In that case $\pi_B$ is irreducible.
\item[(3)] $\,\,\psi_B$ is factorial but not of type $I$ if and only if
	$B(I-B)$ is not of trace class.  In that case
\begin{itemize}
  \item[(i)] $\,\,\psi_B$ is of type $II_1$ if and only if $B-tI$ is
	Hilbert--Schmidt where $0 < t < 1$; then $\pi_B$ is a factor
	representation of type $II_1$.
  \item[(ii)] $\,\,\psi_B$ is of type $II_\infty$  if and only if
	(a) $B(I-B)(B-pI)^2$ is trace class where $0 < t < 1$ and (b) the
	essential spectrum of $B$ contains $0$ or $1$; then $\pi_B$ is a factor
	representation of type $II_\infty$.
  \item[(iii)] $\,\,\psi_B$ is of type $III$ if and only if $B(I-B)(B-pI)^2$
	is not of trace class whenever $0 < t < 1$; then $\pi_B$ is a factor
        representation of type $III$.
\end{itemize}
\end{itemize}
Similar considerations hold for $SU(\infty)$, $SO(\infty)$ and $Sp(\infty)$.
\smallskip

In \cite{W8} we will examine the case where the inducing representation
$\tau$ is a unitary character on $P$. In the finite dimensional
case that leads to a $K$--fixed vector, spherical functions on $G$
and functions on the symmetric space $G/K$.

\end{document}